\documentclass[11pt, english]{article}

\usepackage[latin9]{inputenc}
\usepackage{amstext}
\usepackage{amsthm}
\usepackage{amssymb}
\usepackage{soul}
\usepackage{xfrac}
\usepackage{amsmath}
\usepackage{bm}
\usepackage{enumitem}
\usepackage[dvipsnames]{xcolor}

\makeatletter

\numberwithin{equation}{section}
\numberwithin{figure}{section}

\newtheorem{main}{Main Theorem}
\usepackage{amsthm}
\usepackage{fullpage}
\usepackage{makeidx}
\usepackage{amsfonts}
\usepackage{latexsym}
\numberwithin{equation}{section}

\numberwithin{figure}{section}

\newtheorem{thm}{Theorem}[section]
\newtheorem{prop}[thm]{Proposition}
\newtheorem{lemma}[thm]{Lemma}
\newtheorem{defn}[thm]{Definition}
\newtheorem{coro}[thm]{Corollary}
\newtheorem{rem}[thm]{Remark}

\usepackage{babel}

\newcommand{\Z}{\mathbb{Z}}

\newcommand{\C}{\mathbb{C}}

 \global\long\def\h{\mathfrak{h}}

 \global\long\def\hh{\hat{\mathfrak{h}} }
 \global\long\def\n{\mathfrak{n}}
\usepackage{babel}

\usepackage[blocks]{authblk}% The option is for
\title{  On irreducibility of modules of Whittaker type: twisted modules and  nonabelian orbifolds }

\author{Dra\v{z}en Adamovi\'c, Ching Hung Lam,  Veronika Pedi\' c Tomi\' c and Nina Yu}

\usepackage{babel}

\makeatother

\begin{document}
\maketitle

{\centering\footnotesize {\em Dedicated to Masahiko Miyamoto on the occasion of his 70th birthday} \par}

\abstract{In \cite{ALPY},  we extended  the Dong-Mason theorem on irreducibility of modules for cyclic orbifold vertex algebras (cf. \cite{DM}) to the entire category weak modules and applied this result to Whittaker modules.  In this paper, we present further generalizations  of these results for nonabelian orbifolds of vertex operator superalgebras.
%and  extend the Dong-Mason theorem on irreducibility of modules for orbifold vertex algebras (cf. \cite{DM}) to the category of modules of Whittaker type for vertex  superalgebras.
Let $V$  be a vertex  superalgebra with a countable dimension and let $G$ be a finite subgroup of $\mathrm{Aut}(V)$. Assume that $h\in Z(G)$ where $Z(G)$ is the center of the group $G$.  For any irreducible $h$--twisted (weak) $V$--module $M$, we prove that if $M\not\cong g\circ M$ for all $g\in G$  then $M$ is also irreducible as $V^G$--module. We also apply this result to examples and give irreducibility of modules of Whittaker type for  orbifolds of  Neveu-Schwarz vertex superalgebras, Heisenberg vertex algebras, Virasoro vertex operator algebra and Heisenberg-Virasoro vertex algebra.

\vspace{.5em}
\vspace{.5em}

\section{Introduction}

\subsection*{Quantum Galois theory}Quantum Galois theory
for vertex operator algebras  has been initiated by C. Dong and
G. Mason in \cite{DM}. Let $U$ be a vertex operator algebra and $g$ {an automorphism of $U$ with finite order.}  It is proved in
\cite{DM} that if $M$ is an irreducible ordinary $U$--module such that it is not isomorphic to $
 g^{i}\circ M$ for all $i,$ then $M$
is also an irreducible module for the subalgebra $U^{\left\langle g\right\rangle } $ of fixed points under {the cyclic group} $\left\langle g\right\rangle$ { generated by $g$}. This result is important for the construction of irreducible modules
for some vertex operator algebras. This result is extended in various directions in \cite{DLM,  DM2,DY, MT}.
\subsection*{Results from \cite{ALPY}}
We should mention that all above mentioned approaches are using concepts such as Zhu algebra, higher Zhu algebras, modular invariance, which can not be applied beyond the categories of ordinary or admissible modules. But recent progress in representation theory of infinite-dimensional Lie algebras and superalgebras shows that it is natural to investigate Whittaker modules and other weak modules which are not ordinary modules and on which one cannot apply Zhu algebra or higher Zhu algebras. Therefore in \cite{ALPY} we started a program to extend results from Quantum Galois Theory  to weak modules, so they can be applied beyond the category of ordinary modules.

In \cite{ALPY}, we prove an extension
of Dong-Mason theorem for weak modules of Whittaker type:

 \begin{thm}[\cite{ALPY}] \label{general-intr-whittake-old}
Let $W$ be an irreducible weak $V$--module such that all $W_{i}=   g^{i} \circ W $
are Whittaker modules whose Whittaker functions $\lambda^{(i)}=\mathfrak{n}\rightarrow{\C}$ are mutually distinct.
Then $W$ is an irreducible weak  $V^{\left\langle g\right\rangle }$--module.
\end{thm}

Then we  apply this theorem to Whittaker modules for Heisenberg  and Weyl vertex algebras, and obtain that they are irreducible modules for orbifold algebras. This result provides a new family of Whittaker modules for $W$--algebras.

Let us recall that in the case of Weyl vertex algebra, which we denote by $M$,
% {$M$\color{red} ??},
we construct a family of  irreducible Whittaker modules $M_1 (\bm{\lambda}, \bm{\mu})$ where  $ (\bm{\lambda}, \bm{\mu}) \in {\C} ^n \times {\C} ^n$ and  prove the following result.

%{\color{red}
\begin{thm}[\cite{ALPY}] Assume that $  (\bm{\lambda}, \bm{\mu})  \ne 0$.
% ( {\color{red} $\Lambda$s not needed later ?}).
 Then $M_1(\bm{\lambda}, \bm{\mu})  $ is an irreducible weak  module for the orbifold subalgebra $M^{\Z_p}$, for each $ p \ge 1$.
\end{thm}
Later it was proved by T. Creutzig and A. Linshaw \cite[Theorem 9.4]{CL-triality} that $M^{\Z_p}$ is isomorphic to   $W$-algebra $W_{k} (sl(p), f)$ of level $k =-p+ \frac{p+1}{p}$, corresponding to  the subregular nilpotent element $f$. So we constructed a new family of Whittaker modules for affine $W$--algebras.

 The question
then arises is whether Dong-Mason theorem
{
(or the extension in \cite{ALPY})
}
is true for

\begin{itemize}
\item[(1)] weak twisted modules of a vertex operator superalgebra;
\item[(2)] for  non-abelian orbifolds $V^{G}$, when $G$ is a non-abelian finite group;
\item[(3)] for  orbifolds  $V^{G}$, when $G$ is an infinite-group.
\end{itemize}
In this paper we give a positive answer to generalizations of types (1) and (2).  But we should
 point out that in \cite{AP-22} we proved that $M_1(\bm{\lambda}, \bm{\mu})  $ is reducible for an infinite orbifolds of $M$. So it seems that there are no natural generalizations to infinite groups.

 \subsection*{Main results of the paper}
In this paper, we assume that a vertex  superalgebra $V$ has a countable dimension. We prove
\begin{main}[cf. Theorem  \ref{main1}] Let $V$ be { vertex  superalgebra with a countable dimension}, $G\le \mathrm{Aut}(V)$ a finite group (not necessarily abelian) and $h\in Z(G)$, the center of the group $G$. Let $M$ be an irreducible weak $h$-twisted module of $V$ such that   $ M\not\cong  g\circ M$ for all $g\in G$. Then $M$ is also irreducible as a $V^G$--module.
\end{main}
In the case $g=h$ and $G=\left\langle g\right\rangle$, this theorem gives a twisted version of Theorem   \ref{general-intr-whittake-old}. Another interesting question that arises is the application to non-abelian orbifolds.  In Theorem \ref{slutnja-1}, we prove \cite[Conjecture  8.2]{ALPY}  of irreducibility of Whittaker modules for the permutation orbifolds $M(1)^{S_{\ell}}$  of the Heisenberg vertex algebra.   Since $M(1)^{S_{\ell}}$  is a finitely-generated $W$-algebra, in this way we also construct many examples of irreducible Whittaker modules for $W$-algebras.

\subsection*{Ordinary vs. Whittaker modules in Galois theory}

Let us now discuss similarities and differences between ordinary modules and Whittaker type  modules.

 \begin{itemize}

\item
The theorem of Dong-Mason has been generalized in a much more general setting to any finite group and ordinary modules in \cite{DLM,DY} and to ordinary twisted modules in \cite{MT}.  The proof depends heavily on the theory of Zhu algebras and their generalizations. That the graded subspace $M(n)$ is of finite dimension for each $n\in \Z_+$ is also important for their arguments. For modules of Whittaker type, there are no preferred $\Z_+$--grading and the original proof cannot be applied.
 It turns out that if $W\not\cong  g\circ W$ for all $g\in G$, then a simple  argument using the character  theory of finite groups is sufficient to show that $W$ is irreducible as $V^G$--module.  We also discuss some partial results for the case
that $W\cong  g\circ W$ for all $g\in G$ using some similar arguments.

\item Whittaker modules discussed in this paper and \cite{ALPY} provide a huge family of examples of modules $W$ with  $g \circ W  \not \cong W$, for each $g \in G$ and  $G < \mbox{Aut}(V)$.

\item   Assume first that $W$ is a weak $g$--twisted module for a vertex superalgebra $V$ where $g$ is an automorphism of finite order.
 If $V$ is a vertex operator algebra and $W$ is an ordinary or admissible $V$--module,   then $g \circ W \cong W$, and $W$ is a completely reducible module (cf. \cite{DLM3}).

% {\color{blue}  No all these cases $W$ is an irreducible module for the orbifold vertex superalgebra $V^{\langle g \rangle}$. ( These modules might be reducible for the orbifold vertex superalgebra $V^{\langle g \rangle}$.)}
\item {If $V$ is a vertex operator superalgebra with the canonical automorphism $\sigma$ of order two, then there exists some irreducible ordinary module} $W$ such that $ \sigma \circ W \not\cong W$. In particular, irreducible highest weight modules for Ramond $N=1$ algebra, viewed as Ramond twisted modules for $N=1$ Neveu Schwarz algebra, provide examples of such modules.

\item  The irreducibility result from our paper  can be applied  to simplify proofs of irreducibility of certain modules for orbifold vertex algebra; for example, to show directly that certain twisted modules for the triplet vertex algebra $W(p)$ are irreducible for $W(p) ^{D_{2n}}$--orbifold (cf. \cite{ALM}).
\end{itemize}

%Let us now discuss possible similarity and differences between results

\subsection*{Organization of the paper}
The paper is organized as follows. In Section 2, we recall basic notions and results in vertex  superalgebra. In Section 3, we study the cyclic vectors in a direct sum of irreducible weak twisted modules. In Section 4, we prove the main result of this paper which extends Dong-Mason's theorem on irreducibility of modules for orbifold vertex algebras to the category of Whittaker type for vertex superalgebras. Complete reducibility of certain $V^{\left\langle g\right\rangle }$--modules is studied in Section 5. We apply the main theorem of this paper (Theorem \ref{main1}) to examples in Sections 6, 7, 8 and 9. The irreducibility of twisted Whittaker modules for $N = 1$ Neveu-Schwarz
vertex superalgebras is studied in Section 6. The irreducibility of Whittaker modules for permutation orbifolds of Heisenberg vertex algebra and Virasoro vertex operator algebra is studied in Sections 7 and 8. In Section 9, we study irreducible twisted modules for certain orbifold of Heisenberg-Virasoro vertex operator algebra.

\subsection*{Acknowledgements}
%{\color{blue}
After we finished this work and posted the article on the arXiv, Dong, Ren and Yang posted another article on arXiv \cite{DRY} in which they determine the decomposition of any irreducible (untwisted) $V$--module into irreducible $\C[G]\otimes V^G$--modules for any finite group $G$.  In particular, they established a Schur-Weyl type duality and showed that $(\C[G], V^G)$ forms a dual pair on every irreducible $V$--module. They also proved some
similar results for irreducible $\sigma$--twisted modules if $\sigma$ is in the center of $G$. Later, the same result has been generalized to arbitrary $g$--twisted modules for $g\in G$ by K. Tanabe \cite{Tan}.
%}
%{\color{red}
Their results offer an alternative perspective on our Theorem \ref{main1}, yet they do not explore its applications or further generalizations. While some of their results are more general, our findings were presented earlier. In comparison, our paper offers motivations and examples, and discusses possible extensions of the results to infinite groups, which their work does not address.
%}

  D. Adamovi\' c and V. Pedi\' c Tomi\' c were  partially supported by the
Croatian Science Foundation under the project IP-2022-10-9006 and by the project "Implementation of cutting-edge research and its application as part of the Scientific Center of Excellence QuantiXLie", PK.1.1.02, European Union, European Regional Development Fund.

C. H.  Lam was partially supported by a research grant AS-IA-107-M02 of Academia Sinica and MoST
grants 110-2115-M-001-011-MY3 of Taiwan.
N. Yu was partly supported by the National Natural Science Foundation of China (Grant Nos. 12131018 and 12161141001).

\section{Basics}\label{sec:2}

In this section,  we review some basic notions of vertex  superalgebra
\cite{Bo,DL,FLM} and fix some necessary notations. Throughout this
paper, $z_{0}, z_{1},z_{2}$ are independent commuting formal variables.

\emph{A super vector space} is a $\mathbb{Z}_{2}$--graded vector space
$V=V_{\bar{0}}\oplus V_{\bar{1}}.$ The elements in $V_{\bar{0}}$ (resp.
$V_{\bar{1}}$) are called even (resp. odd). Let $\left|v\right|$
be 0 if $v\in V_{\bar{0}}$ and $1$ if $v\in V_{\bar{1}}.$

\begin{defn}A \emph{vertex superalgebra} is a quadruple $(V, {\bf 1},D,Y)$, where $V=V_{\overline{0}}\oplus V_{\overline{1}}$
is a $\mathbb{Z}_{2}$--graded vector space, $D$ is an endomorphism
of $V$, $1$ is a specified vector called the vacuum of $V$, and
$Y$ is a linear map
\begin{alignat*}{1}
Y\left(\cdot,z\right): & V\to\left(\mathrm{End}V\right)\left[\left[z,z^{-1}\right]\right]\\
 & v\mapsto Y\left(v,z\right)=\sum_{n\in\mathbb{Z}}v_{n}z^{-n-1}(\text{where\ }v_{n}\in\mathrm{End}V)
\end{alignat*}
 such that
\begin{itemize}
	\item[(V1)] For any $u, v\in V, u_{n}v=0$ for $n$ sufficiently large;

\item[(V2)] $\left[D, Y\left(v, z\right)\right]=Y\left(D\left(v\right), z\right)=\frac{d}{dz}Y\left(v, z\right)$
for any $v\in V;$

\item[(V3)] $Y\left({\bf 1}, z\right)=\mathrm{Id}_{V}$ (the identity operator of
$V$);
	\item[(V4)] $Y\left(v, z\right){\bf1}\in\left(\mathrm{End}V\right)\left[\left[z\right]\right]$
and $\lim_{z\to0}Y\left(v,z\right){\bf 1}=v$ for any $v\in V;$

\item[(V5)] For $\mathbb{Z}_{2}$-homogeneous elements $u, v\in V$, the following
Jacobi identity holds:
\begin{gather*}
z_{0}^{-1}\left(\frac{z_{1}-z_{2}}{z_{0}}\right)Y\left(u,z_{1}\right)Y\left(v,z_{2}\right)-\left(-1\right)^{\left|u\right|\left|v\right|}z_{0}^{-1}\delta\left(\frac{z_{2}-z_{1}}{-z_{0}}\right)Y\left(v,z_{2}\right)Y\left(u,z_{1}\right)\\
=z_{2}^{-1}\delta\left(\frac{z_{1}-z_{0}}{z_{2}}\right)Y\left(Y\left(u,z_{0}\right)v,z_{2}\right).
\end{gather*}
\end{itemize}

{
A vertex superalgebra $V$ is called \emph{a vertex operator superalgebra} if there is a distinguished vector $\omega$ of
$V$ such that
\begin{itemize}
\item[(V6)]  $\left[L\left(m\right),L\left(n\right)\right]=\left(m-n\right)L\left(m+n\right)+\frac{1}{12}\left(m^{3}-m\right)\delta_{m+n,}c$
	\ for $m,n\in\mathbb{Z},$ where $Y\left(\omega,z\right)=\sum_{n\in\mathbb{Z}}L\left(n\right)z^{-n-2}.$
	
\item[(V7)] $L\left(-1\right)=D,$ i.e., $\frac{d}{dz}Y\left(v,z\right)=Y\left(L\left(-1\right)v,z\right)$
	for $v\in V;$
	
\item[(V8)] $L(0)$ acts semisimply on $V$ with eigenvalues in $\frac{1}{2}\mathbb{Z}.$
	
\item[(V9)] $\dim V_k<\infty $ and $V_k=0$ for sufficiently small $k$, where $V_k= \{ v\in V\mid L(0) v=kv\}$ for $k\in \frac{1}{2}\mathbb{Z}$.  If $v\in V_k$, $k$ is called the weight (or conformal weight) of $v$ and denoted by $\mathrm{wt}v$.
\end{itemize}
}	
\end{defn}

%In the following, we will {\color{purple} assume a vertex superalgebra  has a countable dimension. Note that all vertex operator superalgebras have countable dimensions.}

\begin{defn}   An
\emph{automorphism} $g$ of
{
 a vertex algebra $V$ is a linear automorphism of $V$
 such that $gY\left(v,z\right)g^{-1}=Y\left(gv,z\right)$
for $v\in V$.

If $V$ is a vertex operator superalgebra, we also assume that an automorphism $g$  preserves the distinguished element $\omega$, i.e., $g\omega =\omega$.
}
\end{defn}

Let $V$ be a vertex operator superalgebra. There is a canonical order 2 linear automorphism $\sigma$ of $V$ associated with the structure of super vector space $V$ such that $\sigma|_{V_{\bar i}}=(-1)^i$ for $i=0, 1$. It is easy to see that $\sigma$ is an automorphism of vertex operator superalgebra $V$.

Let $g$ be an automorphism of $V$
with finite order $T$. Then $V$ is a direct sum of the
eigenspaces $V^{j}$ of $g$,
\[
V=\coprod_{j\in\mathbb{Z}/T\mathbb{Z}}V^{j},
\]
where $V^{j}=\left\{ v\in V\mid gv=\eta^{j}v\right\} $
for $\eta$ a fixed primitive $T$-th root of unity.

\begin{defn}  A \emph{weak g--twisted $V$--module
}is a %\st{$\mathbb Z_2$}-
vector space $M$
%=M^{\bar{0}}\oplus M^{\bar{1}}$
equipped
with a linear map
\begin{align*}
Y_{g}(\cdot,z):\  & V\to({\rm End}\;M)[[z^{1/T},z^{-1/T}]]\\
 & v\mapsto Y_{g}(v,z)=\sum_{n\in\frac{1}{T}\mathbb{Z}}v_{n}^{g}z^{-n-1},
\end{align*}
 %with  $v_{n}^{g}\in\left(\text{End}M\right)^{\left|v\right|}$
  such
that for $u,v\in V$ and $w\in M$ the following conditions hold:

(1) For any $v\in V,\ w\in M$, $v_{n}^{g}w=0$ for $n$ sufficiently
large;

(2) $Y_{g}({\bf 1},z)=1_{M}$ (the identity operator on $M$);

(3) For $u, v\in V$ of homogeneous sign, the following Jacobi identity
holds
\begin{gather*}
z_{0}^{-1}\left(\frac{z_{1}-z_{2}}{z_{0}}\right)Y_{g}(u,z_{1})Y_{g}(v,z_{2})-(-1)^{|u||v|}z_{0}^{-1}\delta\left(\frac{z_{2}-z_{1}}{-z_{0}}\right)Y_{g}(v,z_{2})Y_{g}(u,z_{1})\\
=\frac{z_{2}^{-1}}{T}\sum_{j\in\mathbb{Z}/T\mathbb{Z}}\delta\left(\eta^{j}\frac{\left(z_{1}-z_{0}\right)^{1/T}}{z_{2}^{1/T}}\right)Y_{g}(Y(g^{j}u,z_{0})v,z_{2}).
\end{gather*}
A  weak g--twisted $V$--module $M$  is called a {$\mathbb Z_2$}--graded  if $ M= M^{\bar{0}}\oplus M^{\bar{1}}$ and  $ v_ n ^{g} w \in M^{\bar{i} + \bar{j}}$ for
each $v \in V^{\bar i}, w \in M^{\bar{j}}$ and $n\in\frac{1}{T}\mathbb{Z}$.
\end{defn}
If $g=1$, then a weak $g$--twisted $V$--module is a weak $V$--module.

\begin{rem}
%{\color{blue}
We do not require that a weak g--twisted V--module be $\mathbb Z_2$--graded, primarily due to the examples that will be presented in  Section \ref{ns}.
%}
\end{rem}
\begin{defn}
Let $o(g\sigma)=T'$. An \emph{admissible} $g$--twisted $V$--module is a weak $g$--twisted $V$--module $M$ which carries a $\frac{1}{T'}\mathbb Z_+$-grading $M=\oplus_{n\in\frac{1}{T'}\mathbb Z_+}M(n)$ satisfying $v_mM(n)\subseteq M(n+\mathrm{wt}v-m-1)$ for homogeneous $v\in V$.

\end{defn}

One may assume $M\left(0\right)\not=0$. If $g=1$, then an\emph{
}admissible $g$--twisted $V$--module is called an admissible
$V$--module.

\begin{defn}
		An ordinary \emph{$g$--twisted $V$--module } is
a weak $g$--twisted $V$--module $M$ satisfying the condition that
$M=\oplus_{h\in\mathbb{C}}M_{(h)}$, where $M_{(h)}=\left\{ w\in M\mid L(0)^{g}w=hw\right\}$ for $L\left(0\right)^{g}=\omega_{1}^{g},$
$\dim M_{(h)}<\infty$ for all $h\in\mathbb{C}$, and $M_{(h+\frac{n}{T'})}=0$
for fixed $h\in\mathbb{C}$ and for all sufficiently small integers
$n$. \end{defn}

If $g=1$, then an ordinary $g$--twisted $V$--module is an ordinary
$V$--module.

\medskip
%We say that $V$ is $g$-rational if every admissible $g$--twisted $V$--module is a direct sum of simple admissible $g$--twisted modules. $V$ is called rational if it is $1$-rational.

Let $\text{Aut}\left(V\right)$ be the group of automorphisms of $V$. Now we consider the action of Aut$\left(V\right)$ on twisted modules.
Let $g, h\in\text{Aut}\left(V\right)$ with $g$ of finite order.
If $\left(M,Y_{M}\right)$
is a weak $g$--twisted $V$--module, there is a weak $hgh^{-1}$--twisted
$V$--module $\left(h\circ M,Y_{h\circ M}\right)$ where $h\circ M\cong M$
as a vector space and $Y_{h\circ M}\left(v,z\right)=Y_{M}\left(h^{-1}v,z\right)$
for $v\in V$. This defines a left action of $\text{Aut}\left(V\right)$
on weak twisted $V$--modules and on the isomorphism classes of weak twisted
$V$--modules. We write $$h\circ \left(M,Y_{M}\right)=\left(h\circ M,Y_{h\circ M}\right)=h\circ M.$$
We say $M$ is \emph{$h$--stable} if $M$ and $h\circ M$ are isomorphic.
Note that $g\circ M$ and $M$ are isomorphic if $M$ is an admissible $g$--twisted $V$--module \cite{DLM3}.

%{\dblue What means last sentence?}

%Denote by $\mathfrak{U}\left(g\right)$
%the equivalence classes of irreducible $g$--twisted $V$--modules. Assume that $g$ and %$h$ commute, then $h$ acts on $\mathfrak{U}\left(g\right)$. We denote
%\[
%\mathfrak{U}\left(g,h\right)=\left\{ M\in\mathfrak{U}\left(g\right)|M\circ h\cong %M\right\} .
%\]
}%
%Both $\mathfrak{U}\left(g\right)$ and $\mathfrak{U}\left(g,h\right)$
%are finite sets since $V$ is $g$-rational for all $g$.

\section{On cyclic vectors in a direct sum of irreducible weak $g$--twisted
modules}

Let {$V=(V, {\bf 1},D,Y)$ be  a vertex  superalgebra} and $g$ an automorphism
of  order $T$. Let $t$ be an indeterminate. Recall from \cite{Bo} that $\mathbb{C}\left[t^{\frac{1}{T}},t^{-\frac{1}{T}}\right]$
has the structure of vertex superalgebra with vertex operator
\[
Y\left(f\left(t\right),z\right)g\left(t\right)=f\left(t+z\right)g\left(t\right)=\left(e^{z\frac{d}{dt}}f\left(t\right)\right)g\left(t\right).
\]
Then the tensor product $\mathcal{\mathcal{L}}\left(V\right)=\mathbb{C}\left[t^{\frac{1}{T}},t^{-\frac{1}{T}}\right]\otimes V$ is a vertex superalgebra \cite{DL,FHL,L}
with vertex operator

\[
Y\left(f\left(t\right)\otimes v,z\right)\left(g\left(t\right)\otimes u\right)=f\left(t+z\right)g\left(t\right)\otimes Y\left(v,z\right)u.
\]
The special endomorphism  of $\mathcal{L}\left(V\right)$ is
given by $\mathcal{D}=\frac{d}{dt}\otimes1+1\otimes D$.
The
action of $g$ naturally extends to an automorphism of the tensor
product vertex superalgebra in the following way:
\[
g\left(t^{m}\otimes a\right)=\exp\left(\frac{-2\pi im}{T}\right)\left(t^{m}\otimes ga\right).
\]
Denote the space of $g$--invariant of this action by $\mathcal{L}\left(V,g\right),$ which
is a vertex subalgebra of $\mathcal{L}\left(V\right)$. It is clear
that
\[
\mathcal{L}\left(V,g\right)=\oplus_{r=0}^{T-1}t^{r/T}\mathbb{C}\left[t,t^{-1}\right]\otimes V^{r}.
\]

It follows
from \cite{Bo} that

\[
V\left[g\right]=\mathcal{L}\left(V,g\right)/\mathcal{D}\mathcal{L}\left(V,g\right)
\]
carries the structure of a Lie superalgebra with bracket
\[
\left[u+\mathcal{D}\mathcal{L}\left(V,g\right),v+\mathcal{D}\mathcal{L}\left(V,g\right)\right]=u_{0}v+\mathcal{D}\mathcal{L}\left(V,g\right).
\]

We will use $v\left(n\right)$ to denote the image of $t^{n}\otimes v\in\mathcal{L}\left(V,g\right)$
in $V\left[g\right].$ If $M$ is a weak $g$--twisted $V$--module, then $M$ becomes a $V\left[g\right]$--module
such that $a\left(m\right)$ acts as $a_{m}$.
{
When $V$ is a vertex operator superalgebra,}  $M$ is an
admissible $g$--twisted $V$--module if and only if $M$ is a $\frac{1}{T}\mathbb{Z}_{+}$--graded
module for the graded Lie superalgebra $V\left[g\right]$ \cite{DLM-2,DZ}.

\begin{lemma} \label{ALPY-twisted}
	Let $V$ be a vertex superalgebra with a countable dimension.	
Assume that $W_{i}$, $i=1,\dots,t$,
are non-isomorphic irreducible weak $g$--twisted $V$--modules and
$\overline{W}=\oplus_{i=1}^{t}W_{i}.$ Then for each $0\not=w_{i}\in W_{i},$
a vector of the form $\left(w_{1},w_{2},\dots,w_{t}\right)$ is cyclic
in $\overline{W}.$

\end{lemma}
\begin{proof}
It is clear that $W_{i}$ is an irreducible module for the associative
algebra $\mathcal{U}\left(V\left[g\right]\right).$
The result follows by \cite[Lemma 3.1]{ALPY}.
\end{proof}

\section{Main result}

  	Let $V$ be a vertex superalgebra with a countable dimension.
Now let $G< \mathrm{Aut}(V)$ be a finite subgroup (not necessarily abelian).   Assume that $h \in Z(G)$, where $Z(G)$ is the center of the group $G$.

%Recall the following results.
 %\begin{lemma} \cite[Lemma 3.3] {ALPY} \label{lem:01}
%Assume that $L_i,$  $i=1, \dots, t,$  are   non-isomorphic  irreducible $h$--twisted weak $V$--modules and  $ \mathcal L = \bigoplus_{i=1} ^t L_i$.  Then for each $w_i \ne 0$, $w_i \in L_i$,  a  vector of the form $(w_1, w_2, \dots, w_t )$ is cyclic in $\mathcal L$.
%\end{lemma}

%\medskip

\begin{lemma}[\cite{DM, L}] \label{lem:02}

	Let $V$ be a vertex superalgebra with a countable dimension
and $W$ a {weak} $h$--twisted $V$--module.  Let $S$ be a subset of $W$. Then the submodule  generated by $S$ is given by
$$
\langle S\rangle = \mbox{Span}_{\C} \{ u_n w \ \vert \ u \in V, w\in S , n\in \mathbb{Q} \}.
$$
\end{lemma}

Next we recall  a version of Schur's Lemma given in \cite[Section 4.1.2]{GW}.

\begin{lemma} Let $W_{1}$ and $W_{2}$ be irreducible modules
	for an associative algebra $A$. Assume that $W_{1}$ and $W_{2}$
	have countable dimensions over $\mathbb{C}.$ Then  $$\dim\mathrm{Hom}_{A}\left(W_{1},W_{2}\right) \le 1$$
	and $\dim\mathrm{Hom}\left(W_{1},W_{2}\right)=1$ if and only if $W_{1}\cong W_{2}.$
\end{lemma}

\begin{lemma} \label{countable dimensional}Assume that $V$ is
a vertex superalgebra with a countable dimension. Let $g\in\mathrm{Aut}\left(V\right)$ with finite order. Suppose $W$ is an irreducible weak $g$--twisted $V$--module.
	Then $W$ has a countable dimension.
\end{lemma}
\begin{proof}
	By our assumption,
	 $V$ has a countable dimension.
	By arguments similar to \cite[ Proposition 4.1]{DM},
	we have
	\[
	W=V.w=\text{Span}_{\mathbb{C}}\left\{ v_{n}w\mid v\in V,n\in\frac{1}{T}\mathbb{Z}\right\}\quad \text{ for any } 0\neq w\in W,
	\]
	which implies that $W$ also has a countable dimension.
\end{proof}

Therefore, Schur's lemma also applies to irreducible weak $g$--twisted modules if
$V$ has a countable dimension. From now on, we will assume that $V$ has a countable dimension. Note that this condition is automatically satisfied if $V$ is a vertex operator superalgebra.

\begin{defn}
Let $M$ be a weak $h$--twisted $V$--module. We define the $g$--conjugate  $h$--twisted module $\left( g \circ M , Y_{g \circ M}(\ ,z )\right)$ as follows:
$g \circ M=M$ as  a vector space and
\[
Y_{g \circ M}(u, z) = Y(g^{-1} u,z)\quad \text{ for any } u\in V.
\]
\end{defn}
Note that $  g_1 \circ (g_2\circ M) = (g_1 g_2) \circ M $ for any $g_1, g_2\in \mathrm{Aut}(V)$.
\medskip

Let $G$ be a finite subgroup of $\mathrm{Aut}(V)$. Let $h \in G$ and let $M=(M, Y_M)$ be an irreducible  weak $h$--twisted modules of $V$. Define a subgroup
\[
G_M= \{ g\in G\mid g \circ M \cong M\}.
\]
Note that $G_M$ is a subgroup of $C_G(h)$.

By Schur's lemma and similar arguments as in \cite{DY}, $G_M$ acts on $M$ projectively, that means, there is a projective representation
$h\to \phi(h)$ of $G_M$ on $M$ such that
\[
\phi(h)^{-1}Y_M(v, z)\phi(h) = Y_M(h^{-1}v,z) \quad \text{ for } h\in G_M, v\in V.
\]
Let $\alpha_M$ be the corresponding 2-cocycle in $C^2(G_M, C^*)$. Then $\phi(h)\phi(k) =\alpha_M(h, k)\phi(hk)$ for all $h,k\in G_M$. Without loss, we may assume that $\alpha_M(h,k)$ is a root of unity for any $h,k\in G$.  Let $n$ be the smallest positive integer such that $\alpha_M(h,k)^n=1$ for all $h,k\in G$.  Then we have a  central extension
\[
1\to \langle \xi_n\rangle \to \hat{G}_M\to G_M\to 1
\]
of $G_M$ associated with $\alpha_M$,  where $\xi_n$ is a primitive $n$--th root of unity.

Let $\mathbb{C}^{\alpha_M}[G_M]$ be the corresponding twisted group algebra. It is well known that $\mathbb{C}^{\alpha_M}[G_M]$ is a semisimple associative algebra. It follows that $M$ is a $\mathbb{C}^{\alpha_M}[G_M]$--module.

In this article, we will consider two cases: $G_M=1$ and $G=G_M$. {First we consider the case when $G_M=1$.}

\begin{thm} \label{main1}	
		Let $V$ be a vertex superalgebra with a countable dimension.
	Let $G$ be a finite subgroup (not necessarily abelian) of $\mathrm{Aut}(V)$. Assume that $h\in Z(G)$ where $Z(G)$ is the center of $G$. Let $M$ be an irreducible {weak} $h$--twisted  module of $V$. Suppose that $  M\ncong  g \circ M$ for all $g \in G\setminus{\{1\}}$.
Then $M$ is also irreducible as  $V^G$--module.
\end{thm}

Before we give the proof of the above theorem, we   need to introduce some notions and prove several lemmas.
First we consider the  $h$--twisted  $V$--module
\[
\mathcal{M} = \oplus_{g\in G}  g \circ M.
\]
By our assumption, $\mathcal{M}$ is a direct sum of $|G|$ inequivalent irreducible  $h$--twisted  $V$--modules and $G$ acts naturally (from the left)
on $\mathcal{M}$. Moreover, each $g \circ M$ is isomorphic to $M$ as a $V^G$--module.
In other words,  $\mathcal{M}= M\otimes \C[G]$ as a module of $V^G\times \C[G]$.

For each $g\in G$, let $i_{g}: M \to g \circ M \to \mathcal{M}$ be the natural inclusion.

\begin{defn}
For each $\chi\in \text{Irr}(G)$, define $P_{\chi}\in \C[G]$ by
\[
P_{\chi}= \frac{\chi(1)}{|G|} \sum_{g\in G} \chi(g^{-1}) g.
\]
Denote $P_{\chi}^{V}$ and $P_{\chi}^{\mathcal{M}}$ to be the corresponding actions of $P_{\chi}$ on $V$ and $\mathcal{M}$, respectively.
\end{defn}
Recall that $V= \oplus_{\chi\in \text{Irr}(G)}  V_\chi$ \cite[Corollary 2.5]{DLM}, where $V_\chi$ is the sum of all irreducible
$G$--modules affording the character $\chi$.
%{\color{blue}
Recall that $P_{\chi}^V(V)$ (resp., $P_{\chi}^{\mathcal{M}}(\mathcal{M})$) is exactly the sum of all irreducible
$G$--modules in $V$ (resp., in $\mathcal{M}$) affording the character $\chi$. Thus, we have $V_\chi= P_{\chi}^V(V)$ and
\[
V=\oplus_{\chi\in \text{Irr}(G)} V_\chi= \oplus_{\chi\in \text{Irr}(G)}  P_{\chi}^{V} (V).
\]
Moreover,
\[
\mathcal{M} = \oplus_{\chi\in \text{Irr}(G)} P_{\chi}^{\mathcal{M}} (\mathcal{M})= \oplus_{\chi\in \text{Irr}(G)} \mathcal{M}_\chi ,
\]
where $\mathcal{M}_\chi= P_{\chi}^{\mathcal{M}} (\mathcal{M})$.
%}
Now we have the following result.

\begin{lemma}\label{actionP}
For each character $\chi\in \text{Irr}(G)$ and $u\in V$, $w\in \mathcal{M}$, $m\in \Z$, we have
\[
(P^V_{\chi}(u))_m (P_{1}^{\mathcal{M}}(w)) =  P_{\chi}^{\mathcal{M}} \left(u_m (P_{1}^{\mathcal{M}}(w))\right).
\]
\end{lemma}

\begin{proof}
For $a\in G$,  we have
\[
\begin{split}
\left(P^V_{\chi}(u)\right)_m i_{a}(w) = &\frac{\chi(1)}{|G|} \sum_{g\in G} \chi(g^{-1}) (gu)_m i_{a}(w)\\
=& \frac{\chi(1)}{|G|} \sum_{g\in G} \chi(g^{-1}) i_{a} ((a^{-1}gu)_m w).
\end{split}
\]
Let $h^{-1}=a^{-1}g$. Then $g^{-1}= h a^{-1}$ and we have
\[
\left(P^V_{\chi}(u)\right)_m i_{a}(w) = \frac{\chi(1)}{|G|} \sum_{h\in G} \chi(h a^{-1}) i_{a} \left( (h^{-1}u)_m w\right).
\]
Moreover, we have
\[
\begin{split}
P_{\chi}^{\mathcal{M}} \left(u_m i_{h}(w)\right)  = &\frac{\chi(1)}{|G|} \sum_{g\in G} \chi(g^{-1})g\left( i_{h}\left((h^{-1}u)_m w\right)\right) \\
=& \frac{\chi(1)}{|G|} \sum_{g\in G} \chi(g^{-1}) i_{gh} \left((h^{-1}u)_m w\right)\\
= &\frac{\chi(1)}{|G|} \sum_{a\in G} \chi(h a^{-1}) i_{a} \left( (h^{-1} u)_m w\right).
\end{split}
\]
Therefore,
\[ \
\begin{split}
\left(P^V_{\chi}(u)\right)_m \left(P_{1}^{\mathcal{M}}(w)\right) & = { \frac{1}{|G|}}\sum_{a\in G} \left(P^V_{\chi}(u)\right)_m i_{a}(w) \\
&= \frac{\chi(1)} {|G|^2} \sum_{a\in G} \sum_{h\in G} \chi(ha^{-1}) i_{a} ( (h^{-1}u)_m w)\\
&= {  \frac{1}{|G|}} \sum_{h\in G} P_{\chi}^{\mathcal{M}} (u_m i_{h}(w)) \\
&= P_{\chi}^{\mathcal{M}} \left(u_m (P_{1}^{\mathcal{M}}(w))\right)
\end{split}
\]
as desired.
\end{proof}

\begin{lemma}\label{chi}
For each $\chi\in \text{Irr}(G)$, we have $V_\chi\cdot \mathcal{M}_{1} \subset \mathcal{M}_\chi$.
\end{lemma}

\begin{proof}
It follows from Lemma \ref{actionP}.
\end{proof}
\medskip

\noindent \textbf{Proof of Theorem  \ref{main1}.}  Assume that $M$ is not a simple $V^{G}$--module. Then there is a $V^{G}$--submodule
$0\ne S\subsetneqq M$. Then $P_1^{\mathcal{M}}(S) \subsetneqq P_1^{\mathcal{M}} (M)$.

Now let $s\in S$ be nonzero. Then  $\mathcal{M}$ is generated by $P_1^{\mathcal{M}}(s)$ by Lemma \ref{ALPY-twisted}. By Lemma \ref{lem:02}, we have
\[
\mathcal{M}=V\cdot P_1^{\mathcal{M}}(s) = \oplus_{\chi\in \text{Irr}(G)} V^\chi\cdot P_1^{\mathcal{M}}(s)
\subseteq  \oplus_{\chi\in \text{Irr}(G)} V^\chi\cdot P_1^{\mathcal{M}}(S).
\]
By Lemma \ref{chi},  we have $V^\chi\cdot P_1^{\mathcal{M}}(s) \subset \mathcal{M}_{\chi}$  but $P_1^{\mathcal{M}}(S) \subsetneqq  \mathcal{M}_1$.
Thus $V\cdot P_1^{\mathcal{M}}(s) $ is a proper submodule of $\oplus_{\chi\in \text{Irr}(G)} \mathcal{M}_\chi= \mathcal{M}$, which is a contradiction.
\qed

\medskip

Next we consider the case $G=G_M$.
{
Again, we assume that $V$ is a vertex superalgebra with a countable dimension.} Let $G$ be a finite subgroup of $\mathrm{Aut}(V)$ and $h\in Z(G)$, the center of $G$. Let $M$ be an irreducible  {weak $h$--twisted} module of $V$. Suppose that $  M\cong g\circ  M$ for all $g \in G$, i.e, $G=G_M$.   Let $\alpha_M$ be the corresponding 2-cocycle in $C^2(G, \C^*)$ as described in Section \ref{sec:2}.
Let $1\to\langle \xi_n\rangle \to  \hat{G}\ \bar{\to}\  G\to 1$ be the corresponding central extension of $G$ associated with  $\alpha_M$.
Then $\hat{G}$ acts on $M$ and  we have  $M= \oplus_{ \lambda \in \mathrm{Irr}(\hat{G})}  M^\lambda$, where
$M^\lambda = P_\lambda (M)$.   We use $P^M_\lambda$ to denote the restriction of $P_\lambda$ on $M$ and
$\bar{g}$ to denote the image of $g\in \hat{G}$ under the projection map $\bar{\ }:\hat{G} \to  G$.

\begin{thm} \label{main2}
Let $G$, $\hat{G}$, $h$ and $M$ be as above. Suppose $\lambda$ is a linear character of $\hat{G}$,  i.e., $\lambda(1)=1$.
Then $M^\lambda$ is irreducible as  $V^G$--module if $M^\lambda\neq 0$.
\end{thm}

\begin{proof}
The proof is similar to the proof of Theorem \ref{main1}.
Let $\lambda$ be a linear character of $\hat{G}$. We claim that
\[
\left(P^V_{\chi}(u)\right)_m \left(P_{\lambda}^{M}(w)\right) =  P_{\lambda\chi}^{M} \left(u_m \left(P_{\lambda}^{M}w\right)\right)
\]
for any character $\chi\in \text{Irr}(G)$ and $u\in V$, $w\in M$ and $m\in \Z$.
 Note that an irreducible character $\chi $ of $G$ can be viewed as an irreducible character of $\hat{G}$  by assuming the center $\langle \xi_n\rangle$ acts trivially, i.e,
 $\chi(g)= \chi(\xi_n^i g)=\chi(\bar{g})$ for any $g\in \hat{G}$ and $0 \leq i< n$. Then  we have
 \[
 P^V_{\chi}(u)=\frac{\chi(1)}{|G|}\sum_{s\in G} \chi(s^{-1}) su =
 \frac{\chi(1)}{|\hat{G}|}\sum_{g\in \hat{G}} \chi(g^{-1}) gu.
 \]

By direct calculations, we have
\[
\begin{split}
\left(P^V_{\chi}(u)\right)_m \left(P_{\lambda}^{M}(w)\right)
&=
{
\left(\frac{\chi(1)}{|\hat{G}|}\sum_{g\in \hat{G}} \chi(g^{-1}) gu\right)_m \left ( \frac{\lambda(1)}{|\hat{G}|}
\sum_{h\in \hat{G}} \lambda(h^{-1}) hw \right)
}
\\
&= \frac{\chi(1)}{|\hat{G}|^2} \sum_{g\in \hat{G}} \sum_{h\in \hat{G}} \chi(g^{-1}) (gu)_m  \left(\lambda(h^{-1}) hw\right)\\
&= \frac{\chi(1)}{|\hat{G}|^2} \sum_{g\in \hat{G}} \sum_{h\in \hat{G}} \chi(g^{-1})
\lambda(g^{-1}) g\left( \lambda(h^{-1} g) u_m (g^{-1}h)w\right) \\
&= P_{\lambda \chi}^{M} \left(u_m (P_{\lambda}^{M}(w))\right).
\end{split}
\]
Let $S\subsetneq M^\lambda$ be a proper $V^G$-submodule. Then $V^\chi\cdot S \subset M^{\lambda \chi}$  for any $\chi\in \mathrm{Irr}(G)$ and $V\cdot S$ is a proper submodule of $M$. Therefore, $S=0$ and we have the desired result.
\end{proof}

\section{On complete reducibility of certain $V^{\left\langle g\right\rangle }$--modules}

In this section, we still let $V$ be a { vertex superalgebra with a countable dimension} and $g\in \text{Aut}(V)$ with $o(g)=T$.   Let
\[
V=\coprod_{j\in\mathbb{Z}/T\mathbb{Z}}V^{j},
\]
 where $V^{j}=\left\{ v\in V\mid gv=\eta^{j}v\right\} $ and $\eta$ is  a fixed primitive $T$-th root of unity  as before.

\begin{thm} \label{complete-reducibility-1} Assume that $g$ is an automorphism of $V$ of finite
order $T$. Assume that $W$ is an irreducible weak $g$--twisted $V$--module
such that $g \circ W \cong W$. Then $W$ is a completely reducible
weak $V^{0}$--module such that

(1) $W=\oplus_{i=0}^{T}W^{i}$, $V^{i}.W^{j}\subset W^{i+j\ \text{mod\ } T}$,
where $W^{j}$, $j=1,\cdots,T$,  are eigenspaces of certain linear
isomorphism $\phi\left(g\right):W\to W$.

(2) Each $W^{i}$ is an irreducible weak $V^{0}$--module.

(3) The modules $W^{i},$ $i=0,\cdots,T-1$, are non-isomorphic as
weak $V^{0}$--modules.

\end{thm}
\begin{proof}
(1) and (2) follow from Theorem \ref{main2} by letting $G=G_M=\langle g\rangle$.

For (3),  assume that there is $V^{0}$-isomorphism $\psi:W^{i}\to W^{j},$
for some $i\not=j$. Let $w$ be a nonzero vector in $W^{i}$ and
consider
\[
\mathcal{U}=V.\left(w,\psi\left(w\right)\right)=\text{Span}_{\mathbb{C}}\left\{ \left(v_{n}w,v_{n}\psi\left(w\right)\right)\mid v\in V\right\} .
\]
Then $W^{i}\oplus W^{i}$ is not in $\mathcal{U}$ and hence $\mathcal{U}$
is a proper submodule of $W\oplus W$. Note that $W\oplus W$ is a
$U\left(\mathfrak{g}\left(V\right)\right)$--module of finite length,
the Jordan Holder theorem can be applied. Comparing the filtrations
\[
\left(0\right)\to W\to W\oplus W,\ \ \left(0\right)\to\mathcal{U}\to W\oplus W,
\]
we obtain $\mathcal{U}\cong W$ as $U\left(\mathfrak{g}\left(V\right)\right)$--modules
by the simplicity of $W$. Thus $\mathcal{U}\cong W$ as $V$--modules.
Now both projection maps $\mathcal{U}\to W\oplus\left(0\right)$ and
$\mathcal{U}\to\left(0\right)\oplus W$ are $V$--isomorphisms. Thus
the map $\Phi:v_{n}w\mapsto v_{n}\psi\left(w\right),$ for $v\in V$,
is also an isomorphism. Using Schur's Lemma, we obtain $\Phi=a\mathrm{Id}$
for $a\in\mathbb{C}$ and hence $\psi\left(w\right)=aw\in W^{i}.$ But
$i\not=j,$ which is a contradiction.
\end{proof}

\section{Whittaker modules for  $N=1$ Neveu-Schwarz vertex superalgebras and their orbifolds }
\label{ns}

In this section we shall apply our results to orbifolds  for  $N=1$ Neveu-Schwarz vertex superalgebras $V_c(\mathfrak{ns})$. The Whittaker modules for   $V_c(\mathfrak{ns})$ were studied in \cite{LPX,LPX-ramond}. Their Whittaker modules are irreducible as $\Z_2$--graded modules (see Propositions \ref{LPX} and \ref{LPX-ram} below). Here we identify a family of non $\Z_2$--graded Whittaker modules, on which our theory can be applied.
\vskip 5mm

Recall that the Neveu-Schwarz algebra is the Lie superalgebra
$$ \mathfrak{ns} = \bigoplus _{i \in {\Z}} {\C} L(i) \oplus \bigoplus _{r \in \tfrac{1}{2} + {\Z}} {\C} G(r) \oplus {\C} C $$
which satisfies the following commutation relations:
 \begin{eqnarray}
   [L(m), L(n)] &=& (m-n)L(m+n) + \tfrac{1}{12}(m^3-m)\delta_{n+m,0} C \nonumber  \\ \
   [L(m), G(r)] &=&  (\frac{m}{2} -r ) G(m+r)  \nonumber \\ \
   [G(r), G(s) ] &=& 2 L(r+s) +  (r^2-\tfrac{1}{4})\delta_{r+s,0} C \nonumber \\ \
   [\mathfrak{ns}, C] &=& 0 \nonumber
\end{eqnarray}
for all $m,n \in {\Z}$, $r, s \in {\tfrac{1}{2} + \Z}$.

Let $V_{c} (\mathfrak{ns})$ be the universal Neveu-Schwarz vertex superalgebra of central charge $c \in {\C}$. The category of weak   $V_{c} (\mathfrak{ns})$--modules coincides with the category of restricted $\mathfrak{ns}$--modules of central charge $c$.

Note that $V_{c} (\mathfrak{ns})$ has the canonical automorphism $\sigma$ of order two which is lifted from the automorphism of Lie superalgebra $\mathfrak{ns}$ such that
$$ L(n) \mapsto L(n), G(n+ \tfrac{1}{2}) \mapsto - G(n+\tfrac{1}{2}), C \mapsto C \quad (n \in {\Z}). $$
Let $V_{c}  ^+ (\mathfrak{ns}) $ be the fixed point vertex subalgebra.  It is proved in \cite{HMW, MP1} that $V_{c}  ^+ (\mathfrak{ns}) $  is a $W$--algebra of type $W(2,4,6)$.

Let $$\mathfrak{p} = \bigoplus _{i >0 }   L(i) \oplus \bigoplus _{i > 0} {\C} G(i+\tfrac{1}{2})$$
 and let $\Psi :  \mathfrak{p} \rightarrow {\C}$ be a Lie superalgebra homomorphism.
 Assume that $\Psi$ is non-zero. Then $\Psi$ is uniquely determined by  $(\mu,\nu) \in {\C}^2$, $\mu \cdot \nu \ne 0$ such that
 $\Psi(G(i+\tfrac{1}{2}))  =\Psi(L(i+2))= 0$ for each $i \in {\Z}_{>0}$  and
 $\Psi(L(1))  = \mu  $  and  $\Psi(L(2))  = \nu$. Denote such homomorphism by  $\Psi_{\mu, \nu}$.

 Let ${\C} w$ be the $1$--dimensional $\left(\mathfrak{p} + {\C} C\right)$--module such that
 $$ x w = \Psi_{\mu, \nu} (x) w \quad (x \in \mathfrak{p}), \quad C w = c w.$$
 The universal Whittaker module associated to Whittaker function $\Psi_{\mu, \nu}$ is defined as
 $$Wh (\Psi_{\mu, \nu} , c) =  U(\mathfrak{ns}) \otimes_{U(\mathfrak{p} + {\C} C)} {\C} w.$$
   It is also  an restricted  $\mathfrak{ns}$--module, so it is a  $V_{c} (\mathfrak{ns})$--module.   The following result was proved in \cite{LPX}:
  \begin{prop}\label{LPX} \cite{LPX}  Assume that  $\mu \cdot \nu \ne 0$. Then  $Wh  (\Psi_{\mu, \nu}, c) $
  %= Wh  (\Psi_{\mu, \nu}, c) ^{\bar 0} +  Wh  (\Psi_{\mu, \nu}, c) ^{\bar 1}$
   is an irreducible $\Z_2$--graded $V_{c} (\mathfrak{ns})$--module.
  \end{prop}

  We can write   $\mu = a^2$ for $a\in {\C}$.
Let now  $\widetilde{\mathfrak{p}} =  \mathfrak{p}  \oplus  {\C} G(\tfrac{1}{2})$.
Let $ \widetilde \Psi_{a,\nu}: \widetilde{\mathfrak{p}} \rightarrow {\C}$ be a Lie superalgebra homomorphism uniquely determined by
 $\widetilde \Psi_{a, \nu}(G(i+1/2))  =\widetilde \Psi_{a, \nu}(L(i+2))= 0$ for each $i \in {\Z}_{>0}$  and
$$ \widetilde \Psi_{a, \nu} (G(1/2)) =  a, \widetilde  \Psi(L(1))  = a^2= \mu,  \widetilde \Psi(L(2))  = \nu . $$

 The universal Whittaker module associated to Whittaker function $\widetilde \Psi_{a, \nu}$ is defined as
 $$Wh(\widetilde \Psi_{a, \nu} , c) =  U(\mathfrak{ns}) \otimes_{U(\widetilde{\mathfrak{p}} + {\C} C)} {\C} w.$$
 Let   $ L(\widetilde \Psi_{a, \nu} , c)$ be its simple quotient.

\begin{lemma}  \label{ired-ns-1}  Assume that $\mu \cdot \nu \ne 0$ and $\mu = a^2$. Then
$$ Wh (\Psi_{\mu, \nu} , c)  = Wh(\widetilde \Psi_{a, \nu} , c) + Wh(\widetilde \Psi_{-a, \nu} , c) $$
 and  $Wh(\widetilde \Psi_{a, \nu} , c)$ is an irreducible weak   $V_{c} (\mathfrak{ns})$--module, i.e.  $  L(\widetilde \Psi_{a, \nu} , c) = Wh(\widetilde \Psi_{a, \nu} , c)$
\end{lemma}
 \begin{proof}
Let   $\widetilde L(\widetilde \Psi_{a, \nu} , c)$  by  any  quotient of  $Wh(\widetilde \Psi_{a, \nu} , c)$ with Whittaker vector $w_{a, \nu}$.   Consider the following $\mathfrak{ns}$--module
$$ W = \widetilde L(\widetilde \Psi_{a, \nu} , c) + \widetilde L(\widetilde \Psi_{-a, \nu} , c).$$
Set $$w_0 = w_{a, \nu} + w_{-a, \nu}, w_1 = G(\frac{1}{2}) w_0 = a w_ {a, \nu} - a w_{-a, \nu}.$$
Then $w_0, w_1$ are linearly independent. One gets that $w_0$ is a Whittaker vector with Whittaker function $\Psi_{\mu, \nu}$ and applying Proposition \ref{LPX} we  easily get that  $W = Wh(\Psi_{\mu, \nu} , c)$ is the corresponding universal Whittaker module. This implies  that
$$ Wh(\Psi_{\mu, \nu} , c) =  Wh(\widetilde \Psi_{a, \nu} , c) + Wh(\widetilde \Psi_{-a, \nu} , c)  =  L (\widetilde \Psi_{a, \nu} , c) + L(\widetilde \Psi_{-a, \nu} , c). $$
Thus,   $Wh(\widetilde \Psi_{ a, \nu} , c) =   L (\widetilde \Psi_{a, \nu} , c) $ and the claim holds.
 \end{proof}

 \begin{prop}
%Let  $Wh (\widetilde \Psi_{a,\nu}, c)$ is an irreducible $V_{c} (\mathfrak{ns})$--module for each nonzero Whittaker function $\Psi_{a, \nu}$ as above.  Then $Wh (\widetilde \Psi_{a,\nu}, c)$ is an irreducible $V_{c} ^+(\mathfrak{ns})$--module.
\item[(1)] Let $a \ne 0$.  Then $ (\widetilde \Psi_{a,\nu}, c)$ is an irreducible $V_{c} ^+(\mathfrak{ns})$--module.
\item[(2)]   $L  (\widetilde \Psi_{0,\nu}, c)$ is  a direct sum of two non-isomorphic irreducible modules for $V_{c} ^+(\mathfrak{ns})$.
 \end{prop}
\begin{proof}
   Since $\sigma \circ W (\widetilde \Psi_{a,\nu}, c) =
 Wh (\widetilde \Psi_{-a,\nu}, c)$, we get that $\sigma \circ L  (\widetilde \Psi_{a,\nu}, c) =  L  (\widetilde \Psi_{-a,\nu}, c)  \not \cong   L (\widetilde \Psi_{a,\nu}, c)$ for $a \ne 0$.  Now  the  assertion (1)  follows by applying  Theorem \ref{main1}.

 The case $a = 0$ is different.      Note that $Wh (\widetilde \Psi_{0,\nu}, c)$ is  $\sigma$--invariant,  which implies that  $\sigma \circ L  (\widetilde \Psi_{0,\nu}, c) = L (\widetilde \Psi_{0,\nu}, c)$.
Then Theorem  \ref{complete-reducibility-1} implies that
$$L  (\widetilde \Psi_{0,\nu}, c)  = L^+  (\widetilde \Psi_{0,\nu}, c) \oplus L^-  (\widetilde \Psi_{0,\nu}, c) \, \quad L^{\pm}  (\widetilde \Psi_{0,\nu}, c)   (\Psi, c) = \{ v \in   L^ (\widetilde \Psi_{0,\nu}, c) \ \vert \ \Phi v = \pm v\},$$
where $\Phi$ is a linear isomorphism $L  (\widetilde \Psi_{0,\nu}, c)   \rightarrow L  (\widetilde \Psi_{0,\nu}, c)$. This proves the assertion (2).
\end{proof}

The $N=1$ Ramond algebra is  the Lie superalgebra
$$ \mathfrak{R} = \bigoplus _{i \in {\Z}} {\C} L(i) \oplus \bigoplus _{r  \in {\Z}} {\C} G(r) \oplus {\C} C $$
which satisfies the following commutation relations:
 \begin{eqnarray}
   [L(m), L(n)] &=& (m-n)L(m+n) + \tfrac{1}{12}(m^3-m)\delta_{n+m,0} C \nonumber  \\ \
   [L(m), G(r)] &=&  (\frac{m}{2} -r ) G(m+r)  \nonumber \\ \
   [G(r), G(s) ] &=& 2 L(r+s) +  (r^2-\tfrac{1}{4})\delta_{r+s,0} C \nonumber \\ \
   [\mathfrak{ns}, C] &=& 0 \nonumber
\end{eqnarray}
for all $m,n \in {\Z}$, $r, s \in { \Z}$. The category of weak $\sigma$--twisted $V_{c} (\mathfrak{ns})$--modules coincides with the category of restricted $\mathfrak{R}$--modules of central charge $c$.

Let $$\mathfrak{p}^{tw} = \bigoplus _{i >0 }   L(i) \oplus \bigoplus _{i > 0} {\C} G(i).$$
 Let $\Psi :  \mathfrak{p}^{tw} \rightarrow {\C}$ be the Lie superalgebra homomorphism.
 Assume that $\Psi$ is non-zero. Then $\Psi$ is uniquely determined by  $(\lambda, \mu) \in {\C}^2$, $\mu \cdot \nu \ne 0$ such that
 $\Psi(G(i+1))  =\Psi(L(i+2))= 0$ for each $i \in {\Z}_{>0}$  and
$$ \Psi(L(1))  = \lambda,  \Psi(L(2))  = \mu. $$  Denote such homomorphism by  $\Psi_{\lambda, \mu}$

 Let ${\C} w$ be the $1$-dimensional $\left(\mathfrak{p}^{tw} + {\C} C\right)$--module such that
 $$ x w = \Psi_{\lambda, \mu}(x) w \quad (x \in \mathfrak{p}^{tw}), \quad C w = c w.$$
 The universal Whittaker module associated to Whittaker function $\Psi_{\mu, \nu}$ is defined as
 $$Wh^{tw} (\Psi_{\lambda, \mu}, c) =  U(\mathfrak{R}) \otimes_{U(\mathfrak{p}^{tw} + {\C} C)} {\C} w.$$
%Let  $L ^{tw}(\Psi, c)$ be its simple quotient.

\begin{prop} \label{LPX-ram}  \cite{LPX-ramond}Assume that $\lambda  \cdot \mu  \ne 0$. Then $Wh^{tw} (\Psi_{\lambda, \mu}, c) $ is an irreducible, $\Z_2$--graded $\sigma$--twisted  $V_{c}  (\mathfrak{ns})$--module. \end{prop}

  Let us write   $\mu = b^2$ for $b \in {\C}$.
Let now  $\widetilde{\mathfrak{p}}^{tw} =  \mathfrak{p}^{tw}  \oplus  {\C} G(1)$.
Let $ \widetilde \Psi_{\lambda,b}: \widetilde{\mathfrak{p}}^{tw} \rightarrow {\C}$ be a Lie superalgebra homomorphism uniquely determined by
 $\widetilde \Psi_{\lambda, b}(G(i+1))  =\widetilde \Psi_{\lambda, b}(L(i+2))= 0$ for each $i \in {\Z}_{>0}$  and
$$ \widetilde \Psi_{\lambda, b} (G(1)) =  b, \widetilde  \Psi(L(2))  = b^2,  \widetilde \Psi(L(1))  = \lambda . $$

 The universal Whittaker module associated to Whittaker function $\widetilde \Psi_{\lambda, b}$ is defined as
 $$Wh^{tw} (\widetilde \Psi_{\lambda, b} , c) =  U(\mathfrak{R}) \otimes_{U(\widetilde{\mathfrak{p}}^{tw} + {\C} C)} {\C} w.$$

 As in the proof of Lemma \ref{ired-ns-1} we get

\begin{lemma}  \label{ired-R-1}  Assume that $\lambda \cdot \mu \ne 0$ and $\mu = b^2$. Then
$$ Wh ^{tw} (\Psi_{\lambda, \mu} , c)  = Wh^{tw} (\widetilde \Psi_{\lambda, b} , c) + Wh^{tw}(\widetilde \Psi_{\lambda, -b} , c),  $$
 and  $ L^{tw}(\widetilde \Psi_{\lambda, b} , c)  = Wh^{tw}(\widetilde \Psi_{\lambda, b} , c) $.
 \end{lemma}
 Now applying Theorems \ref{main1}  and    \ref{complete-reducibility-1} we get:
 \begin{prop}
\item[(1)] If $b \ne 0$, then  $L^{tw} (\widetilde \Psi_{\lambda,b}, c)$ is an irreducible $V_{c} ^+(\mathfrak{ns})$--module.
\item[(2)]  $L^{tw} (\widetilde \Psi_{\lambda,0}, c)$  is a direct sum of two non-isomorphic irreducible modules for $V_{c} ^+(\mathfrak{ns})$.
 \end{prop}

\section{  Application to permutation orbifold
{
of Heisenberg vertex algebras}
}
{
In this section, we discuss an application to permutation orbifold
of Heisenberg vertex algebra and give a proof for \cite[Conjecture  8.2]{ALPY}.
}

 We first recall the  definition of Whittaker modules for the Heisenberg vertex algebra $M(1)$ of rank $\ell$ from \cite[Section 8]{ALPY}.
  Let  $\h$ be a complex $\ell$-dimensional vector space with a non-degenerate form $(\cdot, \cdot)$ equipped with the structure of a commutative Lie algebra.  Let $M(1)$ be the Heisenberg vertex operator algebra of rank $\ell$ associated to the Heisenberg algebra ${\hh} = {\h} \otimes {\C}[t,t ^{-1}] + {\C} K$.  Then the automorphism group  of $M(1)$ is isomorphic to the orthogonal group $O(\ell)$. The symmetric group $S_{\ell}$ is a  subgroup of $O(\ell)$
{
and acts on $M(1)$ as permutations on an orthonormal basis of $\mathfrak{h}$.
}

  Let $\n =  {\h} \otimes {\C}[t]  $ be its nilpotent Lie  subalgebra.
  Define the Whittaker function $\bm{\lambda} \in  \n ^*$ such that for each $h \in {\h}$  $$\bm{\lambda} (h(n)) = 0  \quad \mbox{for} \ n >>0.$$
  Let $M(1, \bm{\lambda} )$ be  the standard Whittaker modules for ${\hh}$ of level $1$ associated to the Whittaker function $\bm{\lambda}$.

   For $g \in O(\ell)$ we define:
  $$ g \circ  M(1,\bm{\lambda} )  = M(1,   g \circ\bm{\lambda}). $$

\begin{thm} \label{slutnja-1} Let $G$ be any finite subgroup of $O(\ell)$ such that
$$    g \circ \bm{\lambda}   \ne \bm{\lambda},  \quad \forall g \in G. $$
Then $M(1, \bm{\lambda}) $ is irreducible $M(1) ^G$--module.
\end{thm}
\begin{proof}
The proof follows by applying Theorem \ref{main1}  on modules $ M(1,{g\circ \bm\lambda} ) $ and using the fact that all Whittaker functions $ g\circ \bm{\lambda}  $ are different for $g \in G$.
\end{proof}

As a consequence we prove the following result which was stated as Conjecture 8.2 in \cite{ALPY}.

\begin{coro} Assume that $\sigma\circ\bm{\lambda}\ne\bm{\lambda}$
for any $2$-cycle $\sigma\in S_{\ell}$. Then $M(1,\boldsymbol{\lambda})$
is an irreducible $M(1)^{S_{\ell}}$--module. \end{coro}

\begin{rem}
We believe that $M(1,\boldsymbol{\lambda})$ is  also irreducible for affine $W$-algebras realized as subalgebras of $M(1)^{S_{\ell}}$. But our current  techniques are not sufficient to prove this claim.
\end{rem}

\section{ Permutation orbifold of Virasoro vertex operator algebra}

 Let $Vir$ be the Virasoro Lie algebra with generators $L(n)$, $n \in \Z$, and the central element $C$. Let $Vir^{c}$ be the universal Virasoro vertex operator algebra of central charge $c$. Let $Vir^+$ be the Lie subalgebra of $Vir$ generated by $L(n)$, $n \in {\Z}_{>0}$.

  The classical Whittaker module for the Virasoro algebra  of central charge $c$ is generated by the Whittaker function $\Psi: Vir^+ \rightarrow {\C}$ such that
  $$ \Psi(L(1)) = a, \Psi(L(2)) = b,   \Psi(L(n)) = 0 \quad (n >3). $$
  Let $\text{Wh}_c (\Psi)$  be the universal Whittaker Virasoro module of central charge $c$. Then  $\text{Wh}_c (\Psi)$ is irreducible \cite{LZ}. Since $\text{Wh}_c (\Psi)$ is a restricted $Vir$--module, it is an irreducible module for the $Vir^c$.  Two Whittaker modules $\text{Wh}_c (\Psi_i)$, $i=1, 2$, are isomorphic if and only if Whittaker functions coincide, i.e., $\Psi_1 = \Psi_2$.

Now let $U = Vir^c$, and consider tensor product vertex operator algebra $V = U^{\otimes \ell}$ for $\ell \in {\Z}_{>0}$. Then the symmetric group $S_{\ell}$ acts naturally on $V$
 as a group of automorphisms. Indeed, $\mathrm{Aut}(V) \cong S_{\ell}$ in this case.
Let $G$ be any subgroup of  $S_{\ell}=\mathrm{Aut}(V)$.
By applying Theorem \ref{main1},  we get:
\begin{thm} Assume that   $\Psi_i : Vir^+ \rightarrow {\C}$, $i =1, \dots, \ell$  are Whittaker functions such that
$$ \Psi_i(L(1)) = a_i, \Psi_i(L(2)) = b_i,   \Psi_i(L(n)) = 0 \quad (n >3). $$
Assume that for each $g \in G$:
$$ (a_{g(1)}, \cdots, a_{g(\ell)}) \ne (a_1, \dots, a_{\ell}) \quad \mbox{or} \quad  (b_{g(1)}, \cdots, b_{g(\ell)}) \ne (b_1, \dots, b_{\ell}). $$
Then
$$ \text{Wh}(\Psi_1, \dots, \Psi_{\ell}) :=\text{Wh}_{c} (\Psi_1) \otimes \cdots  \otimes \text{Wh}_{c} (\Psi_{\ell})$$
is an irreducible $V^G$--module.
\end{thm}

%{\color{blue} Maybe we can add some references on permutation orbifolds of Virasoro, Milas-Penn? Then as a consequence, we get Whittaker modules for $W$-algebras.}{\color{red}
The $\mathbb{Z}_3$ and $S_3$-fixed point vertex subalgebras of the tensor product of three copies of the universal Virasoro vertex operator algebras have been studied in \cite{MPS1, MPS2}. It is proved in \cite{MPS1} that for all but finitely many central charges $\left(\mathcal (Vir^{c})^{\otimes 3}\right)^{\mathbb Z_3}$ is a W-algebra of the type $(2,4,5, 6^3, 7, 8^3, 9^3, 10^2)$.
%}

\section{Application to the Heisenberg-Virasoro vertex algebra}
%{\color{blue} In progress}
Recall first the definitions of Heisenberg-Virasoro Lie algebra and mirror twisted Heisenberg-Virasoro Lie algebra.

		The  {\bf  Heisenberg-Virasoro  algebra} $\mathcal H $ is a Lie algebra with a basis
		$$\left\{L(m),h(r), C_1, C_2, C_3, m, r\in{   \Z}\right\}$$
		and commutation relations:
 \begin{eqnarray}
   [L(m), L(n)] &=& (m-n)L(m+n) + \tfrac{1}{12}(m^3-m)\delta_{n+m,0} C_1 \nonumber  \\
  \  [L(m), h(r)] &=&  rh(m+r)+\delta_{m+r,0}(m^2+m) C_2 \nonumber \\
 \    [h(r), h(s) ] &=& r \delta_{r+s, 0} C_3 \nonumber \\
 \   [\mathcal H, C_i] &=& 0 \nonumber \quad i=1,2,3;
\end{eqnarray}
for all $m,n \in {\Z}$, $r, s \in { \Z}$.

The  {\bf   mirror twisted Heisenberg-Virasoro  algebra} $\mathcal H^{tw} $ is a Lie algebra with a basis
		$$\left\{L(m),h(r), \bar C_1, \bar C_2, m, \in{   \Z}, r \in \tfrac{1}{2}+ {\Z} \right\}$$
 and commutation relations:
 \begin{eqnarray}
   [L(m), L(n)] &=& (m-n)L(m+n) + \tfrac{1}{12}(m^3-m)\delta_{n+m,0} \bar C_1 \nonumber  \\
\   [L(m), h(r)] &=&  rh_{m+r}  \nonumber \\
\     [h(r), h(s) ] &=& r \delta_{r+s, 0} \bar C_2 \nonumber \\
\    [\mathcal H, \bar C_i] &=& 0 \nonumber \quad i=1,2;
\end{eqnarray}
for all $m,n \in {\Z}$, $r, s \in \tfrac{1}{2} + { \Z}$.

 Restricted modules (also called smooth modules)  of nonzero level for the mirror Heisenberg-Virasoro   algebra can be treated as  weak twisted modules for the Heisenberg-Virasoro vertex algebras, and restricted modules  of nonzero level  for the Heisenberg-Virasoro   algebra can be treated as  weak modules for the Heisenberg-Virasoro vertex algebras.

 	Let $\mathcal P$ be the subalgebra of $\mathcal H$ spanned by
	$$ C_1, C_2, C_3, L(m),  h(m) , \quad (m \in {\Z}_{\ge 0}). $$
	Let $(\ell_1, \ell_2, \ell_3) \in {\C} ^3$. Consider the $1$-dimensional $\mathcal P$--module ${\C} v $ such that
	$$ C_i  v = \ell_i v, \ i=1,2,3;  \quad   L(m) v   = h(m) v  = 0 \quad (m \in {\Z}_{\ge 0}). $$
	Let $\mathcal V^{\ell_1, \ell_2, \ell_3}$ be the following induced $\mathcal H$--module:
$$\mathcal V^{\ell_1, \ell_2, \ell_3 } = U( \mathcal H ) \otimes_{ U(\mathcal P)} {\C} v . $$	
Then $\mathcal V^{\ell_1, \ell_2, \ell_3}$ is a highest weight $\mathcal H$--module, with the highest weight vector ${\bf 1}  =1 \otimes v$. And $\mathcal V^{\ell_1, \ell_2, \ell_3}$  is a vertex algebra generated by fields
$$ h(z) =   \sum_{r\in {\Z}} h(r) z^{-r-1}, L(z) = \sum_{n\in {\Z}} L(n) z^{-n-2}.$$
This algebra is called the Heisenberg-Virasoro vertex algebra.
Assume that $\ell_3 \ne 0$. It is well known that $M(\ell_3) \cong M(1)$. So it is enough to consider  $\mathcal V^{\ell_1, \ell_2, 1}$. Let $\omega_{a}= \frac{1}{2} h(-1) ^2 + a h(-2)$ be the Virasoro vector in $M(1)$ with central charge $c= 1-12 a^2$. Denote by  $M(1)_a$  the vertex operator algebra $M(1)$ with conformal vector $\omega_a$. We write $M(1)$ for $M(1)_0$ when we choose canonical Virasoro vector $\omega_0 = \frac{1}{2} h(-1) ^2 $ of central charge $c=1$.

The
following theorem is a vertex-algebraic version of the results from \cite{ACKP} (see also
\cite{AJR}, \cite{GuW}):
\begin{thm}
$$ \mathcal V^{\ell_1, \ell_2, 1} \cong M(1)_{\ell_2} \otimes Vir^{c},$$
where  $Vir^{c}$ is the universal Virasoro vertex algebra with conformal vector
$ \omega^{Vir} =  L(-2) - \omega_{\ell_2}$ of central charge $c = \ell_1-1 + 12 \ell_2^2$.
\end{thm}

We shall now consider the case $\ell_2 =0$. Set $\mathcal V = \mathcal V^{\ell_1, 0, 1}$.
The vertex operator algebra $\mathcal V \cong M(1) \otimes Vir^{c}$ has the  automorphism $\sigma$ of order two which is lifted from the automorphism $h \mapsto -h$ of $M(1)$.  We know that
\begin{itemize}
\item  The category of  weak $\mathcal V$--modules is equivalent to the category of restricted $\mathcal H$--modules such that $C_1$ acts as $\ell_1 \mbox{Id}$,  $\ell_1 = c+1$, $C_2$ acts trivially, and $C_3$ acts as $\mbox{Id}$.

\item  The category of  $\sigma$--twisted weak $\mathcal V$--modules is equivalent to the category of restricted  $\mathcal H^{tw}$--modules such that $\bar C_1$ acts as $\ell_1 \mbox{Id}$,   and  $\bar C_2$ acts as $\mbox{Id}$.
\end{itemize}

The fixed point subalgebra is $\mathcal V^+ = M(1) ^+ \otimes Vir^c$. The irreducible  untwisted and twisted  modules for $\mathcal V$ are classified in \cite{TYZ}. We shall now see how some of irreducible $\mathcal V$--modules remain irreducible when we restrict them to $\mathcal V^+$.

Let $p,q\in{\Z}_{>0}$. We consider the subalgebra $\mathcal P^{(p,q)} \subset \mathcal H$ spanned by $L(n), h(m), C_1, C_2, C_3$, $n \ge p$, $m >q$.
   Let   ${\bf a} =(a_1, \dots, a_{q}) \in  ({\C}^*)^{q}$,  ${\bf b} = (b_1, \dots, b_p) \in  ({\C}^*)^p$, $c \in  {\C}$.  Define the $1$-dimensional
$\mathcal P^{(p,q)}$--module $  {\C} v_0$ with
\begin{eqnarray} & C_1\cdot v_0=(c +1) v_0,\,\,\, C_2 \cdot v_0=  0, \ \   C_3 \cdot v_0=  v_0\\
 &L(p)  v_0 = a_1 v_0, \cdots, L(p+q-1) v_0 = a_q v_0, \  L(i) v_0 = 0 \ \mbox{for} \ i >  p+q-1, \\
&h(q+1)  v_0 = b_1 v_0, \cdots   h ( p+q)  v_0 = b_p v_0, \    h_{i} v_0 = 0 \ \mbox{for} \  i > p+q,
 = 0 \ \mbox{for} \  i > p+q.\end{eqnarray}
 Then $ L _c ({\bf a},{\bf b}) = \mbox{Ind}_{\mathcal P^{(p,q)}} ^{\mathcal H} {\C} v_0$ is in irreducible $\mathcal H$--module (cf. \cite{TYZ}).
 Since $ L _c ({\bf a},{\bf b})$ is restricted module, we conclude that it is an irreducible weak module for $\mathcal V$.

 Completely
 {analogous},  we define  the modules $ L^{tw} _c ({\bf a},{\bf b})$ for mirror twisted Heisenberg-Virasoro algebra $\mathcal H^{tw}$. Then $ L^{tw} _c ({\bf a},{\bf b})$
 is an irreducible $\sigma$--twisted $\mathcal V$--module. Applying   Theorem  \ref{main1}	 in this case, we get:

 \begin{thm} Assume that ${\bf b} \ne 0$. Then $ L^{tw} _c ({\bf a},{\bf b})$  and $ L _c ({\bf a},{\bf b})$  are irreducible $\mathcal V^+ = M(1) ^+ \otimes Vir^c$--modules.
 \end{thm}

 It is shown in \cite{TYZ} that $ L^{tw} _c ({\bf a},{\bf b})$ and $ L _c ({\bf a},{\bf b})$ are not tensor product modules for $M(1) \otimes Vir^c$ (see also \cite{AP-22}). As a consequence we get:
 \begin{coro}
  $ L^{tw} _c ({\bf a},{\bf b})$ and $ L _c ({\bf a},{\bf b})$ are not tensor product modules for $M(1)^+ \otimes Vir^c$.
 \end{coro}

\vskip10pt {\footnotesize{}{ }\textbf{\footnotesize{}D.A.}{\footnotesize{}:
Department of Mathematics, Faculty of Science,  University of Zagreb, Bijeni\v{c}ka 30,
10 000 Zagreb, Croatia; }\texttt{\footnotesize{}adamovic@math.hr}{\footnotesize \par}

\textbf{\footnotesize{}C. L.}{\footnotesize{}: Institute of Mathematics,
Academia Sinica, Taipei 10617, Taiwan; }\texttt{\footnotesize{}chlam@math.sinica.edu.tw}{\footnotesize \par}

\textbf{\footnotesize{}V.P.T.}{\footnotesize{}: Department of Mathematics, Faculty of Science,
University of Zagreb,   Bijeni\v{c}ka 30, 10 000 Zagreb, Croatia; }\texttt{\footnotesize{}vpedic@math.hr}{\footnotesize \par}

\textbf{\footnotesize{}N.Y.}{\footnotesize{} School of Mathematical
Sciences, Xiamen University, Fujian, 361005, China;} \texttt{\footnotesize{}
ninayu@xmu.edu.cn}{\footnotesize \par}

{\footnotesize{}}}{\footnotesize \par}

\end{document}